\documentclass[a4paper,10pt,leqno,openright,oneside]{amsart}
\usepackage[latin1,utf8x]{inputenc}
\usepackage{forest}

\usepackage{amsthm,amssymb,amsmath,amsopn,amsfonts,pb-diagram,footmisc}
\usepackage{cite}
\usepackage{stmaryrd,calc,enumerate,mathrsfs,amscd,wasysym,footnote}
\usepackage{graphicx}
\usepackage{fancyhdr}
\usepackage{indentfirst}
\usepackage{geometry}
\usepackage{mathtools}
\usepackage{bbm}
\usepackage{float}

\theoremstyle{plain}
\newtheorem{theorem}{Theorem}[section]
\newtheorem{cor}[theorem]{Corollary}
\newtheorem{lem}[theorem]{Lemma}
\newtheorem{prop}[theorem]{Proposition}

\newtheorem{thm}{Theorem} % Teorema numerat

\theoremstyle{definition} 
\newtheorem{defn}[theorem]{Definition}
\newtheorem{exmpl}{Example}

\newcommand*{\bigchi}{\mbox{\Large$\chi$}}

\DeclareMathOperator{\supp}{supp}

\DeclareMathOperator{\sgn}{sgn}

\title{Some remarks on the Dirichlet problem on infinite trees}
\author{Nikolaos Chalmoukis, {Matteo Levi}}
\thanks{The first author is supported by the fellowship INDAM-DP-COFUND-2015  ''INdAM Doctoral Programme
	in Mathematics and/or Applications
	Cofunded by Marie Sklodowska-Curie Actions'', Grant Number 713485.}
\date{}
\subjclass[2010]{Primary: 31C15. Secondary: 05C63, 05C05}

\address{N. Chalmoukis \\ Dipartimento di Matematica \\ Universit\'a di Bologna \\ 40127 Bologna, Italy}
\email{nikolaos.chalmoukis2@unibo.it}

\address{M. Levi \\ Dipartimento di Matematica \\ Universit\'a di Bologna \\ 40127 Bologna, Italy}
\email{matteo.levi2@unibo.it}

\begin{document}
	
	\maketitle
	
	\begin{abstract}
	We consider the Dirichlet problem on infinite and locally finite rooted trees, and we prove that the set of irregular points for continuous data has zero capacity. We also give some uniqueness results for solutions in Sobolev $ W^{1,p} $ of the tree.
	\end{abstract}
	
	\section*{Introduction}

	A tree $ T $ is a connected graph without loops. In this note we further assume that a tree has a distinguished edge $ \omega $ called the root, that is locally finite (meaning that no vertex is connected with infinitely many other vertices) and has no leaves. By leaf we mean a vertex $ x $ which is the endpoint of a unique edge $ \alpha  $, $ \alpha \neq \omega $. We call $ V(T) $ the vertex set end $ E(T) $ the edge set of the tree.
	
	Trees have become a subject of interest because, due to their relatively simple structure, they can be used as a toy model for complicated situations arising in the study of problems of real as well as complex analysis. This point of view has been adopted in a variety of ways and it has been proven fruitful, for example in the study of the Dirichlet space of analytic function in the unit disc. (See for example \cite{Arcozzi05} \cite{Arcozzi10}). Here we are interested in the tree as an object per se. In particular, in analogy with the Potential Theory in the Euclidean space, we investigate the interplay between the capacity of the boundary of a tree and the Dirichlet problem. 
	
	The paper is organized as follows. In  Section \ref{Potential} we introduce a Nonlinear Potential Theory on $ T $ which allows us to define equilibrium measures, equilibrium potentials and the capacity $ c(E) $ of a subset $ E $ of the boundary $ \partial T $. As in the continuous case, irregular points are points where the equilibrium potential fails to attain the value 1 [Definition \ref{Irregular}].  In Section \ref{Dirichlet} show that in fact such points can be identified by a Wiener's like series of capacities [Theorem \ref{Wieners}]. Furthermore, we show that the classical probabilistic solution of the Dirichlet problem with continuous boundary data actually converges to the given data, except at irregular points [Theorem \ref{SolutionDirichlet}]. As a corollary we get the corresponding Kellogg's Theorem, i.e. the set of irregular points for the  Dirichlet problem on a tree has capacity zero. 
	
	Section \ref{Uniquness} is devoted to the study of uniqueness of the solution of the Dirichlet problem. Here the results are more rudimentary for general trees [Example \ref{CounterExample}]. We introduce Sobolev spaces on a tree and discuss some results about their boundary values.
	
	\section{Preliminaries}\label{Potential}
	
	Our approach on tree capacities is in the framework of an abstract Potential Theory that can  be found for example in \cite{Hedberg99}. We give a brief exposition of the theory in the particular case of trees here. First let us introduce a piece of notation. A geodesic $ \{\alpha_i\} $ is a (finite or infinite) sequence of edges such that  for every $ \alpha_j\in \{\alpha_i\} $, $ \lbrace\alpha_0,\dots,\alpha_j\rbrace $ is the shortest path between $ \alpha_0 $ and $ \alpha_j $. Notice that for every edge $ \alpha $ there exists a unique geodesic $ \lbrace\alpha_0=\omega,\alpha_1,\dots,\alpha_N=\alpha\rbrace=:[\omega,\alpha] $ which starts at the root and ends at $ \alpha $. We call $ |\alpha|:=N $  the level of $ \alpha $. A rooted tree has also a natural partial order attached. For two edges $ \alpha,\beta $, $ \alpha \geq \beta $ if $ [\omega,\alpha] \supseteq [\omega,\beta] $. Hence it makes sense to define successor sets $ S(\alpha)=\{\beta\in E(T): \beta \geq \alpha \} $ and predecessor sets $ P(\alpha)=\{\beta \in E(T):  \beta \leq \alpha \} $, as well as the set of sons of $ \alpha $, $ s(\alpha):=\{\beta \in E(T), \beta \geq \alpha, |\beta|=\alpha+1 \} $ and the (unique) parent of $ \alpha $ denoted $ p(\alpha) \in E(T)$ which satisfies $ \alpha\in s(p(\alpha)) $. We call $ T_\alpha $ the subtree of $ T $ rooted at $ \alpha $ which has as vertex the set $ S(\alpha) $. The boundary of a tree $ T $ is defined as the set of infinite geodesics with starting point $ \omega $, and has a topology generated by the basis $ \{\partial T_\alpha\}_{\alpha\in E(T)} $ where $ \partial T_\alpha $ is the set of infinite geodesics passing through $ \alpha $. It turns out that this space is metrizable and $ T\cup \partial T $ is a compactification of $ T $ with the edge counting metric. Note that the order relation extends naturally to an order on the set $ E(T)\cup V(T) \cup \partial T $.
	
	Let $g$ be a real valued function defined on the vertices if $T$. Given a point $\zeta=\lbrace x_j\rbrace_{j=1}^\infty\in\partial T$, we define the radial limit of $g$ at $\zeta$ as
	\begin{equation*}
	\lim_{x\to\zeta} g(x)=\lim_{j\to\infty} g(x_j).
	\end{equation*}
	The \textit{Fatou's set} of $g$ is
	\begin{equation*}
	\mathcal{F}(g)=\lbrace \zeta \in\partial T: \ \text{there exists} \ \lim_{x\to\zeta} g(x)\in \mathbb{R}\cup\{ \pm \infty \} \rbrace.
	\end{equation*}
	The \textit{boundary value} of $g$ is the map $g^*:\partial T\to \mathbb{R}$ which on Fatou's points $\zeta\in\mathcal{F}(g)$ is defined by
	\begin{equation*}
	g^*(\zeta):=\lim_{x\to\zeta} g(x).
	\end{equation*}
	From now on we simply write $g$ in place of $g^*$ for the extension of $g$ to the boundary since no confusion can arise.
	
	In order to define a capacity of a set $ E\subseteq \partial T $ we need the notion of the potential of a function. 
	
	\begin{defn}
		Suppose that $ f: E(T) \rightarrow \mathbb{R} $. We define its potential $ If: V(T) \rightarrow \mathbb{R} $,
		\begin{equation*}
		If(x)=\sum_{\alpha < x }f(\alpha).
		\end{equation*}
		The potential extends to a function from $V(T)\cup \partial T$ to $\mathbb{R}\cup\{\pm \infty \}$,
		\begin{equation*}
		If(\zeta)=\sum_{\alpha < \zeta }f(\alpha), \quad \text{for} \ \zeta\in\mathcal{F}(If).
		\end{equation*}
	\end{defn}
	It is clear that if $f$ is a nonnegative function then $If$ is defined on the all boundary, possibly taking value $+\infty$,  while in general its Fatou's set is non trivial.
	
	Let  $ p\in (1,+\infty) $ be a fixed exponent and $ p' $ its H\"older conjugate, $ 1/p+1/p'=1 $.
	\begin{defn}
		Suppose that $ E \subseteq \partial T $ is a Borel set. We define the $ p-$capacity of $ E $, 
		\begin{equation*}
		c_p(E):=\inf \{ \Vert f \Vert_{\ell^p}^p: f:E(T)\rightarrow [0,\infty), \ If \geq 1 \,\,\, \text{on} \,\,\, E \},
		\end{equation*} where $ \Vert f \Vert_{\ell^p}^p := \sum_{\alpha \in E(T)}|f(\alpha)|^p $.
	\end{defn}
	Some remarks are in order. It is customary to say that a property holds $ p- $capacity almost everywhere or $ c_p $ almost everywhere if the set on which it does not hold has $ p-$capacity zero. With this terminology one can prove \cite[Theorem 2.3.10]{Hedberg99} that given a Borel set $ E \subseteq \partial T $ there exists a unique function $ f^E:E(T)\rightarrow [0,+\infty) $, such that $ If=1, c_p-$a.e. on $ E $ and $ \Vert f^E \Vert_{\ell^p}^p=c_p(E) $. This function is called the $ p- $\textit{equilibrium function} for the set $ E $. 
	\begin{defn}\label{Irregular}
		A point $ \zeta \in E $ such that $ If^E(\zeta)\neq 1 $ is called \textit{irregular} for the set $ E $.
	\end{defn}
	
	There exists a quite useful equivalent definition of capacities in terms of measures. We call \textit{charge} a signed finite Borel measure on $ \partial T $. The co-potential of a charge $ \mu $ is defined by 
	\begin{equation*}
	I^*\mu(\alpha)=\mu(\partial T_\alpha), \,\,\, \alpha \in E(T).
	\end{equation*}
	For brevity we shall write $ M $ instead of $ (I^*\mu) $ when the implied charge is clear from the context. The $ p$-energy of a charge is just
	\begin{equation*}
	\mathcal{E}_p(\mu)=\Vert M\Vert_{p'}^{p'}.
	\end{equation*}
	We define also the mutual energy of a charge $ \mu $ and a function $ f $ on edges admitting boundary values $\mu-$almost everywhere to be
	\begin{equation*}
	\mathcal{E}(\mu,f) = \int_{\partial T} Ifd\mu.
	\end{equation*}
	If the mutual energy is finite, we can switch sums and integrals and write $\mathcal{E}(\mu,f) = \sum_{\beta\in E(T)}f(\beta)M(\beta)$. For functions on edges we use the footnote notation $f_p:\alpha\mapsto f(\alpha)^{p'-1}$, where powers of negative quantities has to be intended as follows: $a^s:=\sgn(a)|a|^s$, for each $a\in\mathbb{R}$, $s>0$. Hence, if $\mathcal{E}_p(\mu)<\infty$, we can switch sums and integrals and get
	\begin{equation}\label{mixed energy}
	\mathcal{E}(\mu,M_p)=\sum_{\beta\in E(T)}\sgn M(\beta)|M(\beta)|^{p'-1}M(\beta)=\sum_{\beta\in E(T)}|M(\beta)|^{p'}=\mathcal{E}_p(\mu).
	\end{equation}
	
	The following is what is usually called the dual definition of capacity.
	
	\begin{thm}\label{DualDef} \cite[Theorem 2.5.3] {Hedberg99}
		Suppose that $ E \subseteq \partial T $ Borel. Then
		\begin{equation*}
		c_p(E)=\sup\{\mu(E)^p: \mu\geq 0, \supp(\mu) \subseteq E, \mathcal{E}_p(\mu) \leq 1 \}.
		\end{equation*}
		Moreover, there exists a unique positive charge $ \mu^E $ supported in $ E $, called the $ p-$equilibrium measure of $ E $, such that 
		\begin{equation*}
		\mu^E(E)=c_p(E)=\mathcal{E}_p(\mu^E),
		\end{equation*}and $ I^*\mu^E=f^E $.
	\end{thm}

	\section{On p-harmonic functions on trees}\label{Harmonic}
	
	If $g:V(T)\to \mathbb{R}$ is a function of the vertices, we define its \textit{gradient} on the edges to be the difference operator,
	\begin{equation*}
	\nabla g(\alpha)=g(e(\alpha))-g(b(\alpha)),
	\end{equation*}where $ b(\alpha), e(\alpha) $ denote the beginning and the end vertex of $ \alpha $, with respect to the order relation.
	It is immediate to see that the following fundamental theorem of calculus holds.
	\begin{prop}\label{fundamental}
		Take two functions $f:E(T)\to \mathbb{R}$, $g:V(T)\to \mathbb{R}$. Then, $g=If+g(o)$ on $V(T)$ if and only if $f=\nabla g$ on $E(T)$. Where $ o=b(\omega) $.
	\end{prop}
	\begin{proof}
		Let $g=If+g(o)$ on $V(T)$. Then, for every $\alpha\in E(T)$ we have
		\begin{equation*}
		\nabla g(\alpha)=g(e(\alpha))-g(b(\alpha))=\sum_{\beta\leq \alpha}f(\beta)-\sum_{\beta\lvertneqq \alpha}f(\beta)=f(\alpha).
		\end{equation*}
		Vice versa, let $f=\nabla g$ on $E(T)$. Then $If(o)=0=g(o)-g(0)$ and for every $x\in V(T)\setminus \lbrace o\rbrace$, let $\alpha\in E(T)$ be the unique edge such that $x=e(\alpha)$. Then we have
		\begin{equation*}
		If(x)=If(e(\alpha))=\sum_{\beta\leq \alpha}\nabla g(\beta)=\sum_{\beta\leq \alpha}g(e(\beta))-g(b(\beta))=g(x)-g(o).
		\end{equation*}
	\end{proof}
	%As an immediate consequence one has that $I\nabla$ and $\nabla I$ act as identity operators on the space of functions on the vertices and edges, respectively.
	
	We say that a function $f:E(T)\to\mathbb{R}$ is \textit{forward additive} if, for every $\alpha\in E(T)$,
	\begin{equation}\label{forward}
	f(\alpha)=\sum_{\beta\in s(\alpha)}f(\beta).
	\end{equation}
	
	It is immediate that the potential of a charge defines a forward additive function. 
	Next proposition characterizes forward additive functions that can be obtained as potentials of charges.
	
	\begin{prop}\label{charge and forward additive functions}
		A forward additive function  $f:E(T)\to \mathbb{R}$ satisfies
		\begin{equation}\label{condition}
		\lim\limits_{k}\sum_{|\alpha|=k}|f(\alpha)|<\infty,
		\end{equation}
		if and only if there exists a (unique) charge $\mu$ on $\partial T$ such that $f=I^*\mu$.
	\end{prop}
	\begin{proof}
		Note that the limit in condition (\ref{condition}) is in fact a supremum. For $f$ forward additive we have
		\begin{equation*}
		\sum_{|\alpha|=k+1}|f(\alpha)|=\sum_{|\alpha|=k}\sum_{\beta\in s(\alpha)}|f(\beta)|\geq\sum_{|\alpha|=k}|f(\alpha)|,
		\end{equation*}
		from which it follows
		\begin{equation*}
		\sup_k\sum_{|\alpha|=k}|f(\alpha)|=\lim_{k\to\infty}\sum_{|\alpha|=k}|f(\alpha)|.
		\end{equation*}

		For each $\alpha$ write $\zeta(\alpha)$ for an arbitrary point in $\partial T_{\alpha}$. For each $k\in\mathbb{N}$, define a Borel measure,
		\begin{equation*}
		\mu_k=\sum_{|\alpha|=k}f(\alpha)\delta_{\zeta(\alpha)}.
		\end{equation*}
		The total variation of the measure $\mu_k$ is given by
		\begin{equation*}
		\Vert \mu_k\Vert=\sum_{|\alpha|=k}|f(\alpha)\delta_{\zeta(\alpha)}(\partial T)|=\sum_{|\alpha|=k}|f(\alpha)|,
		\end{equation*}
		from (\ref{condition}) it follows that the family of measures $\mu_k$ is uniformly bounded, so that it has a weak$^*$-limit $\mu$ which is positive. For each edge $\alpha$ we have
		\begin{equation*}
		I^*\mu(\alpha)=\mu(\partial T_\alpha)=\int_{\partial T}\bigchi _{\partial T_\alpha}d\mu=\lim_k\int_{\partial T}\bigchi _{\partial T_\alpha}d\mu_k=\mu_k(\partial T_\alpha)=f(\alpha).
		\end{equation*}
		For the uniqueness part, if $f=I^*\nu$ for some other charge $\nu$, then $\nu(\partial T_{\alpha})=\mu(\partial T_{\alpha})$ for each $\alpha\in E(T)$ and hence $\mu\equiv\nu$.
		\bigskip
		
		Viceversa, let $\mu$ be a charge on $\partial T$ and consider the forward additive function $f=I^*\mu$. Then
		\begin{equation*}
		\sum_{|\alpha|=k}|f(\alpha)|=\sum_{|\alpha|=k}|\mu^+(\partial T_\alpha)-\mu^-(\partial T_\alpha)|\leq \sum_{|\alpha|=k}|\mu|(\partial T_\alpha)=\Vert\mu \Vert<\infty.
		\end{equation*}

	\end{proof}

Observe that if $ f\geq 0 $ then condition \ref{condition} is automatically satisfied.

	Given $g:V(T)\to\mathbb{R}$, its $p$-\textit{Laplacian} at the vertex $x$ is given by
	\begin{equation*}
	\Delta_p g (x):=\sum_{y\sim x}\Big(g(y)-g(x)\Big)^{p-1}.
	\end{equation*}
	We say that $g$ is $p$-\textit{harmonic} if $\Delta_p g \equiv 0$ on $V(T)\setminus\lbrace o \rbrace$. As usual, we simply call Laplacian the linear operator $\Delta:=\Delta_2$ and we say that $g$ is harmonic if $\Delta g \equiv 0$. Observe that harmonicity coincide with the mean value property
	\begin{equation*}
	\Delta g \equiv 0  \iff g(x)=\frac{\sum_{y\sim x}g(y)}{\#\lbrace y\in V(T): \ y\sim x\rbrace}, \qquad\text{for all} \ x\in V(T)\setminus\lbrace o \rbrace.
	\end{equation*}
	
	For more details on $ p-$harmonic functions on trees and their boundary behaviour see for example \cite{Canton01}.
	
	\begin{prop}\label{p-harmonic potential}
		A function $f:E(T)\to\mathbb{R}$ is forward additive if and only if $If_p$ is a $p$-harmonic function on $V(T)$.
		%Moreover, given a $p$-harmonic function $g:V(T)\to \mathbb{R}$, then $g=If_p$ for some forward additive function $f:E(T)\to\mathbb{R}$.
	\end{prop}
	\begin{proof}
		Let $x\in V(T)\setminus\lbrace o \rbrace$ and $\alpha\in E(T)$ such that $x=e(\alpha)$. Since $(p'-1)(p-1)=1$, we have
		\begin{equation*}
		\begin{split}
		\Delta_p If_p(x)&=\Big(If_p(b(\alpha))-If_p(e(\alpha))\Big)^{p-1}+ \sum_{\beta\in s(\alpha)}\Big(If_p(e(\beta))-If_p(b(\beta))\Big)^{p-1}\\
		&=-f_p(\alpha)^{p-1}+\sum_{\beta\in s(\alpha)}f_p(\beta)^{p-1}=-f(\alpha)+\sum_{\beta\in s(\alpha)}f(\beta).
		\end{split}
		\end{equation*}
		It follows that $\Delta_p If_p\equiv 0$ if and only if (\ref{forward}) holds.
		%\bigskip
		%
		%
		%
		%Vice versa, given a $p$-harmonic function $g$, define $f:=(\nabla g)_{p'}$. By Remark \ref{notation}, $f_p=\nabla g$ and hence $If_p=g$. Then, the potential of $f_p$ is $p$-harmonic and hence $f$ is forward additive.
	\end{proof}
	
	Putting together the last two Propositions we get the following.
	
	\begin{cor}\label{cor condition 2}
		A $p$-harmonic function $g$ satisfies
		\begin{equation}\label{condition2}
		\sup_k\sum_{|\alpha|=k}|\nabla g(\alpha)|^{p-1}<\infty,
		\end{equation}
		if and only if there exists a charge $\mu$ such that $g=IM_p$.
	\end{cor}
	\begin{proof}
		A function $g$ satisfies (\ref{condition2}) if and only if $f:=\left(\nabla g\right)_{p'}$ satisfies (\ref{condition}) and by Proposition \ref{p-harmonic potential} $g=If_p$ is $p$-harmonic if and only if $f$ is forward additive. By Proposition \ref{charge and forward additive functions} we have the claim.
		
		%If $g$ is $p$-harmonic and satisfies (\ref{condition2}), the function $f=\left(\nabla g\right)_{p'}$ satisfies (\ref{condition}). Moreover, $If_p=g$, which by Proposition \ref{p-harmonic potential} is forward additive. Hence by Proposition \ref{charge and forward additive functions}, there exists a charge such that $f=I^*\mu=M$, which implies $g=IM_p$. Now, assume $g=IM_p$ where $M=I^*\mu$ for some charge $\mu$. By Proposition \ref{p-harmonic potential} $g$ is $p$-harmonic and $\left(\nabla g\right)_{p'}=M$, which satisfies (\ref{condition}) by Proposition \ref{charge and forward additive functions}. Hence
	\end{proof}
	
	%ENERGIA DI UNA FUNZIONE
	%We call $p$-\textit{energy} of the function $f:E(T)\to \mathbb{R}$ the $p$ power of its $\ell^p$ norm, and we write
	%\begin{equation*}
	%\mathcal{E}_p(f)=\Vert f\Vert_{\ell^p}^p.
	%\end{equation*}

	%
	%\textbf{Notation}: whenever we work on some subtree of a given tree we use subscripts to indicate which is the root of the subtree we are referring to. For example, give a rooted tree $T$ and some $\alpha\in E(T)$, we write $I_{\alpha}$ for the potential operator acting on functions defined on the edges of the tree $T_{\alpha}$ having root $\alpha$: $I_{\alpha}f=I(\bigchi_{\partial T_{\alpha}}f)$. Similarly, $c_{p,\alpha}$ will be the capacity when we consider $T_{\alpha}$ as ambient space.
	\section{The Dirichlet problem} \label{Dirichlet}
	
	There is an extensive literature on the discrete Dirichlet problem and its variations on graphs (See for example \cite{Woess00}, \cite{Lyons17} and \cite{Kaimanovich90}). In the particular case of trees we derive more precise results about the exceptional set.

	\subsection{Wiener's test}
	
	For any edge $\alpha$ in a tree $T$, we denote by $c_{\alpha,p}$ the $p$-capacity referred to the tree $T_\alpha$. Given a set $E\subseteq \partial T$, we define $E_\alpha:=E\cap\partial T_\alpha$. The following Theorem can be seen as the analogous for trees of the classical Wiener test for irregular points (see \cite[Theorem 7.1]{Garnett05})

	\begin{theorem} \label{Wieners}
		A boundary point $\zeta$ is irregular for a set $E\subseteq \partial T$ of positive capacity if and only if
		\begin{equation}\label{wiener test}
		\sum_{\alpha < \zeta} c_{\alpha,p}(E_\alpha)^{p'/p}<\infty.
		\end{equation}
	\end{theorem}

In the proof we shall need the following rescaling property of equilibrium measures on trees (see \cite{Levi18}).

\begin{lem}\label{Rescaling}
	Let $\mu$, $\mu_\alpha$ be the $p$-equilibrium measures for the sets $E\subseteq \partial T$ and $E_\alpha$, respectively. Then it holds the relation
\begin{equation*}
\mu |_{\partial T_\alpha}=\Big(1-IM_p(b(\alpha))\Big)^{p/p'}\mu_{\alpha}.
\end{equation*}
	
	Moreover, for every $\alpha\in E(T)$, $\mu$ solves the following equation:
\begin{equation*}
M(\alpha)\Big(1-IM_p(b(\alpha))\Big)=\sum_{\beta\geq\alpha}M(\beta)^{p'}.
\end{equation*}
\end{lem}

	\begin{proof}[Proof of Theorem \ref{Wieners}]
		Let $\mu$ be the equilibrium measure for $E$, and $M$ its co-potential. Set $\varepsilon:=1-IM_p(\zeta)\geq 0$ to be the deficit of regularity of the point $\zeta\in E$. Let $\lbrace\alpha_j\rbrace=P(\zeta)$, and set $t_n=\sum_{j\geq n}M_p(\alpha_j)$. Clearly $t_n$ is monotonically decreasing to zero, being the tail of the converging sum $IM_p(\zeta)$. By Lemma \ref{Rescaling},
		\begin{equation*}
		c_n:=c_{\alpha_n,p}(E_{\alpha_n})^{p'/p}=\frac{M(\alpha_n)^{p'/p}}{1-IM_p(b(\alpha_n))}=\frac{M_p(\alpha_n)}{\varepsilon+t_n}=\frac{t_n-t_{n+1}}{\varepsilon+t_n}.
		\end{equation*}
		Now, the sum $\sum_n c_n$ converges if and only if  $\prod_n(1-c_n)>0$. The partial product can be explicitly calculated thanks to its telescopic structure,
		\begin{equation*}
		\prod_{n=0}^N(1-c_n)=\prod_{n=0}^N\frac{\varepsilon+t_{n+1}}{\varepsilon+t_n}=\frac{\varepsilon+t_{N+1}}{\varepsilon+t_0}.
		\end{equation*}
		Since $t_0=\mu(\partial T)^{p/p'}>0$, it follows that $\prod_{n=0}^\infty(1-c_n)>0$ if and only if $\varepsilon>0$, which is, if and only if the point $\zeta$ is irregular.
	\end{proof}
	
	Observe that the Wiener condition (\ref{wiener test}) can be re-written purely in terms of capacities on the all boundary, in the following sense.
	\begin{cor}
		A boundary point $\zeta$ is irregular for a set $E\subseteq \partial T$ of positive capacity if and only if
		\begin{equation*}
		\sum_{\alpha < \zeta} \frac{c_p(E_\alpha)^{p'/p}}{1-|\alpha|c_p(E_\alpha)^{p'/p}}<\infty.
		\end{equation*}
	\end{cor}
	\begin{proof}
		Let $T$ be any rooted tree, $E\subseteq\partial T$ and consider a tent $T_\alpha$, with $|\alpha|=n$. If $\mu$ is the equilibrium measure for $E_\alpha=E\cap\partial T_\alpha\subseteq \partial T$, then the associated co-potential $M$ is supported on $E(T_\alpha)\cup P(b(\alpha))$, since the equilibrium function $M_p$ must minimize the $p$-norm. By forward additivity, $M$ must be constant on $P(\alpha)$, i.e. $M(\beta)=M(\omega)=c_p(E_\alpha)$ for $\beta\leq\alpha$. By the rescaling properties we know that
		\begin{equation*}
		M(\alpha)=c_{p,\alpha}(E_\alpha)\Big(1-IM_p(b(\alpha))\Big)^{p/p'},
		\end{equation*}
		and since $IM_p(b(\alpha))=nc_p(E_\alpha)^{p'-1}$, we obtain
		\begin{equation*}
		\begin{split}
		c_p(E_\alpha)&=\mathcal{E}_p(\mu)\\
		&=nc_p(E_\alpha)^{p'}+\mathcal{E}_{p,\alpha}(\mu)\\
		&=nc_p(E_\alpha)^{p'}+\Big(1-nc_p(E_\alpha)^{p'/p}\Big)^p c_{p,\alpha}(E_\alpha),
		\end{split}
		\end{equation*}
		from which follows
		\begin{equation*}
		c_{\alpha,p}(E_\alpha)=\frac{c_p(E_\alpha)}{\Big(1-nc_p(E_\alpha)^{p'/p}\Big)^{p/p'}}.
		\end{equation*}
		Substituting this expression in the Wiener condition (\ref{wiener test}) we get the result.
	\end{proof}
	
	\subsection{A probabilistic interpretation of capacity}

	The connection between random walks on graphs and electrical networks is nothing new (see for example \cite{Lyons17}). Here we give an interpretation of the capacity of the boundary of a tree, which will be helpful in the solution of the Dirichlet problem. As usual we work on a general rooted locally finite tree $ T $ without leaves. 
	
	Consider the simple random walk $ (Z_n) $ on the vertices of $ T $  which stops when it hits the root vertex $ b(\omega) $. In this context we consider  $ o=b(\omega) $ part of the extended boundary $ \overline{\partial T} = \partial T \cup \{ o \} $ of $ T $. Then there exists a $ \overline{\partial T}-$ valued random variable $ Z_\infty $ such that $ Z_n $ converges to $ Z_\infty $ , $ \mathbb{P}_x $-almost surely for every $ x\in T $, where $ \mathbb{P}_x $ is the probability measure of the random walk starting at $ x\in T $. We can now associate to any vertex $ x\in T $ the harmonic measure $ \lambda_x(E):=\mathbb{P}_x(Z_\infty \in E),  $ where $ E $ is a Borel subset of $ \overline{\partial T} $.

	\begin{prop}
		Let $ T $ be any tree. Then the $ 2-$capacity of $ \partial T $ equals the probability that a simple random walk starting at $ e(\omega) $ will escape to  the boundary before hitting $ o $. Formally,
		\begin{equation*}
		\lambda_{e(\omega)}(\partial T)=c(\partial T).
		\end{equation*}
	\end{prop}
	
	\begin{proof}
		Suppose that we have a finite tree of depth $ N>0 $. We can naturally identify $ \partial T $ with the vertices of $ T $ with maximal degree. The by the Markov property the function $ h(x)=:\lambda_x(\partial T) $ is harmonic in $ T^\circ:=V(T)\setminus \overline{\partial T} $, $ h(\zeta)=1 $ if $ \zeta \in \partial T $ and $ h(o)=0 $. Since the same is true for the equilibrium function $ If^{\partial T} $ of the boundary of $ \partial T $ by the maximum principle $ If^{\partial T}=h $ and the result follows for finite trees.
		
		For a general tree $ T $ not necessarily finite , let $ T_n $ be the truncation of $ T $ up to level $ n $. Then from the finite case we have that $ c_{T_n}(\partial T_n)= \mathbb{P}_{e(\omega)}(\sup_i|Z_i| \geq n) $. By monotonicity of measures the last quantity converges to $ \mathbb{P}_{e(\omega)}(Z_\infty \in \partial T) $, as $ n \rightarrow +\infty $. It remains to show that $  c_{T_n}(\partial T_n)\rightarrow c_{T}(\partial T) $, as $ n\rightarrow \infty $. By definition of capacity we get that $  c_{T_n}(\partial T_n) \geq c_T(\partial T) $, since the equilibrium function $ f^{\partial T_n} $ is an admissible function for $ \partial T $.
		
		To prove the other inequality we use the dual expression for capacity.  Suppose that $ \mu_n $ is a measure on $ \partial T_n $ such that $ \mathcal{E}(\mu_n)= \sum_{|\alpha| \leq n}M_n(\alpha)^2 \leq 1 $ and $ \mu_n^2(\partial T_n) = c_{T_n}(\partial T_n)$. Consider now the corresponding charges on $ \partial T $ 
		\begin{equation*}
		\widetilde{\mu}_n:=\sum_{|\alpha|=n}\mu_n(\alpha)\delta_{\zeta_\alpha},
		\end{equation*} where $ \zeta_\alpha $ any point in $ \partial T_\alpha $ and $ \delta_{\zeta_\alpha} $ the corresponding Dirac mass. Since $ \widetilde{\mu_n}(\partial T)\leq 1 $ we can find a weak$^*$-limit point $ \mu $ of the sequence $ \{\widetilde{\mu_n} \} $. We have that 
		\begin{equation*}\label{EnergyEstimate}
		\begin{split}
		\sum_{|\beta| \leq m }M(\beta)^2 &  = \lim_{n}\sum_{|\beta| \leq m }M_n(\beta)^2  \\
		& = \lim_{n}\sum_{|x| \leq m }\widetilde{M}_n(\beta)^2 \\
		& \leq \lim_{n}\sum_{|\beta| \leq n }M_n(\beta)^2 \\
		& \leq 1.
		\end{split}
		\end{equation*}  Therefore, letting $ m\rightarrow \infty $ we get that $ \mathcal{E}(\mu)\leq 1 $, and hence by the dual definition of capacity,
		\begin{equation*}
		c_T(\partial T) \geq \mu(\partial T)^2 = \lim\limits_{n}\widetilde{\mu}_n(\partial T)^2 = \lim\limits_{n}c_{T_n}(\partial T_n).
		\end{equation*} 
	\end{proof}

	Given a function $\varphi$ defined on $\partial T$ we define its \textit{harmonic extension} (or its \textit{Poisson integral}) to be the function $\mathcal{P}(\varphi):V(T)\to\mathbb{R}$ given by
	\begin{equation*}
	\mathcal{P}(\varphi)(x):=\int_{\partial T}\varphi \ d\lambda_x.
	\end{equation*}
	The harmonicity of $\mathcal{P}(\varphi)$ follows by the Markov property, since
	\begin{equation*}
	\lambda_x=\sum_{y\sim x}p(x,y)\lambda_y=\frac{1}{\deg(x)+1}\sum_{y\sim x}\lambda_y.
	\end{equation*}
	
	\begin{theorem} \label{SolutionDirichlet}
		For any given $\varphi\in C(\partial T)$, the Poisson integral of $ \varphi $ satisfies
		\begin{equation*}
		\begin{split}
		\begin{cases}
		\Delta \mathcal{P}(\varphi)=0 \,\,\, \text{in} \, V(T)\setminus \{b(\omega)\} \\
		\lim\limits_{x \rightarrow \zeta}\mathcal{P}(\varphi)(x)=\varphi(x), \,\,\, \text{if} \,\,\, \zeta \,\,\, \text{is a regular point of} \,\,\, \overline{\partial T}.
		\end{cases}
		\end{split}
		\end{equation*}
	\end{theorem}
	\begin{proof}
		Pick $\alpha\in E(T)$ and let $\zeta=\lbrace x_j\rbrace_{j=0}^\infty\in \partial T_\alpha$. Write $\lbrace\alpha_j\rbrace_{j=1}^\infty=P(\zeta)$, so that $x_j=e(\alpha_j)$. For $n\geq |\alpha|$ we have
		\begin{equation*}
		\begin{split}
		0\leq 1-\lambda_{x_n}(\partial T_\alpha)&\leq\mathbb{P}_{x_n}(Z_n \,\, \text{hits} \ \alpha \ \text{before hitting} \ \partial T_\alpha)\\
		&=\prod_{j=|\alpha|}^{n} \mathbb{P}_{x_j}(Z_n \,\, \text{hits} \ x_{j-1} \ \text{before hitting} \ \partial T_{\alpha_j})\\
		&=\prod_{j=|\alpha|}^{n}\big( 1-c_{2,\alpha_j}(\partial T_{\alpha_j}) \big).
		\end{split}
		\end{equation*}
		By the Wiener condition (\ref{wiener test}) we have that the right handside vanishes as $n\to +\infty$ if and only if $\zeta$ is a regular point for $\partial T$. Hence, for any regular point $\zeta\in \partial T_\alpha$ we have $\lim_{n\to\infty}\lambda_{x_n}(\partial T_\alpha)= 1$, from which it follows that for any regular point $\zeta$ in the boundary
		\begin{equation*}
		\lim_{n\to\infty}\lambda_{x_n}(\partial T_\alpha)\geq\delta_\zeta (\partial T_\alpha).
		\end{equation*}
		Any open set $E\subseteq\partial T$ can be written as a disjoint union of tents $\lbrace\partial T_{\alpha_k}\rbrace$. Let $\zeta$ be a regular point of  regular of $\partial T$. Then,
		\begin{equation*}
		\liminf_n\lambda_{x_n}(E)=\liminf_n\sum_k\lambda_{x_n}(\partial T_{\alpha_k})\geq \sum_k\liminf_n\lambda_{x_n}(\partial T_{\alpha_k})\geq \sum_k\delta_\zeta (\partial T_{\alpha_k})=\delta_\zeta(E).
		\end{equation*}
	It follows that $\lambda_{x_n}\stackrel{w^*}{\longrightarrow}\delta_\zeta$, as $n\to\infty$. Therefore, for any $\varphi\in C(\partial T)$, 
		\begin{equation*}
		\mathcal{P}(\varphi)(x)=\int_{\partial T}\varphi \ d\lambda_x\longrightarrow \varphi(\zeta), \quad \text{as} \ x\to\zeta.
		\end{equation*}
	\end{proof}
	
	\begin{cor}[Kellog's Theorem for Trees]
		The set of irregular points for the Dirichlet problem has capacity zero. Which is,
		\begin{equation*}
		c\big( \{\zeta \in \partial T: \,\,\, \text{there exists} \,\, \varphi\in C(\partial T), \lim\limits_{x\rightarrow \zeta}\mathcal{P}(\varphi)(x) \neq \varphi (\zeta) \} \big) = 0.
		\end{equation*}
	\end{cor}
	
	\section{Uniqueness results}\label{Uniquness}
	
	We give a first uniqueness result for the class of \textit{spherically symmetric trees}, which are trees where the degree is constant on levels. Clearly homogeneous trees belong to this class.
	
	We define the Lebesgue measure on $\partial T$ to be the measure $\lambda$ which is equidistribuded among sons of any edge and is normalized with $\lambda(\partial T)=1$. Namely, for each $\beta\in E(T)$, if we write $p(\beta)$ for the edge parent of $\beta$, we have
	\begin{equation*}
	\lambda(\partial T_\beta)=\lambda(\partial T_{p(\beta)})/\deg\left(e(p(\beta))\right)=1/\prod_{\alpha<\beta}\deg(e(\alpha)).
	\end{equation*}
	It is clear that on spherically symmetric trees $I^*\lambda$ is constant on levels. In what follows, we write $\lambda(k)$ in place of $\lambda(\partial T_\beta)$ when $|\beta|=k$. One can check that the equilibrium measure of a spherically symmetric tree is a scalar multiple of the Lebesgue measure.

	\begin{prop}\label{spherical}
		Suppose $ T$ is a spherically symmetric tree, with $ c(\partial T)>0 $ and let $ \mu $ be a charge on $ \partial T $. Denote by $M$ its potential. If $ IM=0$ Lebesgue almost everywhere on the boundary, then $ \mu \equiv 0 $. 
	\end{prop}
	
	\begin{proof}
		
		For a fixed $\alpha\in E(T)$, let $s(\alpha)=\lbrace\alpha_j \rbrace_{j=1}^{\deg(\alpha)}$ and define the measures $ \lambda^{\alpha_j}$ on $ \partial T$ in the following way 
		
		\begin{equation*}
		I^*\lambda^{\alpha_j}(\gamma)= \begin{cases}
		\deg(\alpha) I^*\lambda(\gamma) \,\,\, &\text{if} \,\,\, \gamma \geq \alpha_j  \\
		0  &\text{if} \,\,\, \gamma \geq \alpha_i, \ i\neq j \\
		I^*\lambda(\gamma),  &\text{otherwise}.
		\end{cases}
		\end{equation*}
		
		It is clear that $ \lambda^{\alpha_j}$ is absolutely continuous with respect to $ \lambda $. Integrating on the tent rooted in $\alpha$, using the fact that $\lambda(\partial T_\beta)$ depends only on the level of $\beta$, for each $j$ we get
		\begin{equation*}
		\begin{split}
		0&=\int_{\partial {T}_\alpha} IM d\lambda^{\alpha_j}\\
		&= \int_{\partial {T}_\alpha} \sum_{\beta \in E(T)} M(\beta) \chi_{\partial {T}_\beta }(\zeta)d\lambda^{a_j}(\zeta) \\
		&= \sum_{\beta\in E(T)} M(\beta) \lambda^{\alpha_j}(\partial T_\alpha \cap \partial T_\beta) \\
		&= \lambda(\partial T_\alpha)\sum_{\beta < \alpha } M(\beta) +\deg(\alpha) \sum_{\beta \geq \alpha_j} \lambda(\partial T_\beta)M(\beta) \\
		&= 	\frac{1}{2^{|\alpha|}}\sum_{\beta < \alpha } M(\beta) + \deg(\alpha) \sum_{k=|\alpha_j|}^\infty \lambda(\partial T_\beta)\sum_{\substack{\beta \geq \alpha_j, \\ |\beta|=k}} M(\beta)\\
		&=\frac{1}{2^{|\alpha|}}\sum_{\beta < \alpha } M(\beta) + \deg(\alpha) M(\alpha_j)\sum_{k=|\alpha|+1}^\infty \lambda(k)
		\end{split}
		\end{equation*}
		
		Note that the last quantity is finite because the capacity of the boundary is positive. Being the same true for each $j$, $M$ must be constant on $s(\alpha)$. It follows that $\mu=M(\omega)\lambda$, i.e. the measure $\mu$ is a scalar multiple of the Lebesgue measure. Hence,
		\begin{equation*}
		0=\int_{\partial T} IM d\mu=\sum_{\beta\in E(T)}M(\beta)^2,
		\end{equation*}
		which gives $M\equiv 0$ on $E(T)$ which is the thesis.

	\end{proof}
	
	The same is not true for a general tree. In fact, there exists a sub-dyadic tree $ T $, with no irregular boundary points and a charge $ \mu $ on $ \partial T $ such that $ IM=0 $ everywhere except at a point, but $ \mu \neq 0 $, as shown in the next example.
	
	\begin{exmpl} \label{CounterExample}
		The following diagram represents an infinite  subdyadic tree and the copotential $ M $ of a charge $ \mu $ on its boundary, where the number $ r $ over an edge indicates how many times the edge is repeated. Also, the label $ T(a) $ means that the vertex is the root vertex of a dyadic tree which caries a total measure of $ a $ on the boundary, and the measure $ M $ is divided equally at each edge.

		
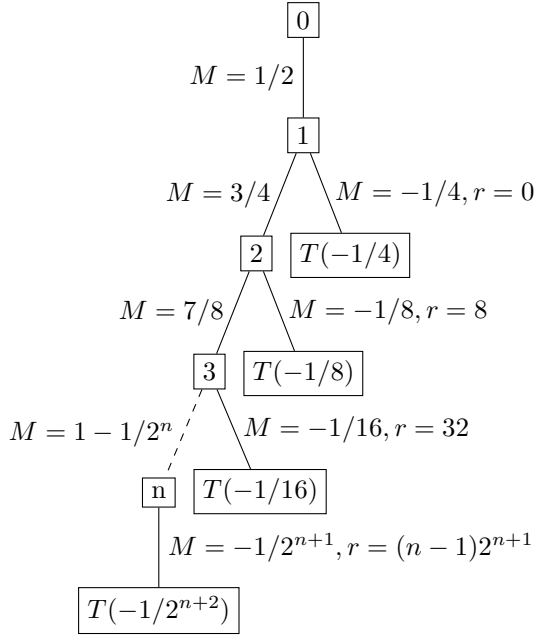
\begin{center}
			\begin{forest}
				for tree={draw, l sep=30pt}
				[0,
				[1,edge label={node[midway,left] {$ M=1/2 $}} 
				[2,edge label={node[midway,left] {$ M=3/4 $}}  
				[3,edge label={node[midway,left] {$ M=7/8 $}}  
				[n, edge label={node[midway,left] {$ M=1-1/2^n $}} , edge={dashed} [ $ T(-1/2^{n+2}) $,edge label={ node[midway,right] {$ M=-1/2^{n+1}, r=(n-1)2^{n+1} $}} ]  
				]
				[ $ T(-1/16) $, edge label={node[midway,right] {$ M=-1/16, r=32 $}}  ]
				] 
				[$ T(-1/8) $,  edge label={node[midway,right] {$ M=-1/8, r=8 $}}] 
				]
				[$ T(-1/4) $, edge label={node[midway,right] {$ M=-1/4, r=0 $}}] 
				] ]
			\end{forest}
		\end{center}
		
		If $ \zeta_0 $ is the leftmost point of the boundary it is clear that $ IM(\zeta_0)=+\infty $. If $ \zeta\in \partial T \setminus \{ \zeta_0 \} $ let $n=\max \lbrace\zeta\cap\zeta_0\rbrace$, where the intersection of boundary points is intended as the intersection of the corresponding geodesics. Then,
		
		\begin{equation*}
		\begin{split}
		IM(\zeta)&=\sum_{i=1}^{n}\big( 1-2^{-i} \big)-\frac{(n-1)2^{n+1}}{2^{n+1}}-\frac{1}{2^{n+1}}\sum_{i=0}^\infty 2^{-i}\\
		&=n-\sum_{i=1}^{n}2^{-i}-(n-1)-2^{-n}=0.
		\end{split}
		\end{equation*}
		
		By applying Wiener's test we see that the $ \zeta_0 $ is a regular point of the boundary, while all other points are clearly regular by symmetry.

	\end{exmpl}

	\subsection{Energy conditions}
	The situation is different if we work with measure with finite energy.

	\begin{prop}\label{same potentials}
		Let $\mu$, $\nu$ be charges on $\partial T$, with $\mathcal{E}_p(\mu)$, $\mathcal{E}_p(\nu)<\infty$. Denote by $M$ and $V$ their potentials , respectively. If $IM_p=IV_p$ both $\mu$-a.e. ad $\nu$-a.e., then $\mu\equiv \nu$.
	\end{prop}
	\begin{proof}
		Integrating both the potentials $IM_p$ and $IV_p$ with respect to both the measures, we can get any of the following
		\begin{equation*}
		\mathcal{E}_p(\mu)=\mathcal{E}_p(\nu,M_p)=\mathcal{E}_p(\mu,V_p)=\mathcal{E}_p(\nu).
		\end{equation*}
		Recalling that we write $a^s=\sgn(a)|a|^s$, with some algebra we get
		\begin{equation*}
		\begin{split}
		0&=\mathcal{E}_p(\mu)+\mathcal{E}_p(\nu)-\mathcal{E}_p(\nu,M_p)-\mathcal{E}_p(\mu,V_p)\\
		&=\sum_{\beta\in E(T)}|M(\beta)|^{p'}+|V(\beta)|^{p'}-V_p(\beta)M(\beta)-M_p(\beta)V(\beta)\\
		&=\sum_{\beta\in E(T)}(M-V)(M_p-V_p)(\beta).
		\end{split}
		\end{equation*}
		It is clear that $a^s-b^s$ has the same sign as $a-b$ for every $a,b\in\mathbb{R}$, $s>0$. It follows that the general term of the above series is positive, from which $M\equiv V$ on $E(T)$. 
	\end{proof}
	
	It is clear from the dual definition of capacity that if a property holds $c_p$-a.e. then it also holds $\mu$-a.e. with respect to any charge $\mu$ of finite energy. Hence, the above result can be rewritten in the following slightly less general but more natural form.
	
	\begin{cor}\label{same potentials cor}
		Given charges $\mu$, $\nu$ on $\partial T$, with $\mathcal{E}_p(\mu)$, $\mathcal{E}_p(\nu)<\infty$, if $IM_p=IV_p$, $c_p$-a.e., then $\mu\equiv \nu$.
	\end{cor}
	
	As a consequence, we have a partial converse of the properties of equilibrium measures given in Theorem \ref{DualDef}.
	
	\begin{cor}\label{converse equilibrium measure}
		Let $\mu$ be a Borel measure on $\partial T$ such that $\mathcal{E}_p(\mu)<\infty$ and $IM_p=1$ $c_p$-a.e. on $E=\supp(\mu)$. Then $\mu$ is the $p$-equilibrium measure for $E$.
	\end{cor}
	
	These uniqueness results can be re-interpreted in terms of functions in place of measures. The following Sobolev space naturally arises from the space of charges of finite energy,
	\begin{equation*}
	W^{1,p}(T):=\lbrace g:V(T)\to\mathbb{R}: \ \nabla g\in \ell^p\rbrace.
	\end{equation*}
	
	In fact, given a charge $\mu$ on $\partial T$, $\mathcal{E}_p(\mu)=\Vert M_p\Vert_p^p=\Vert \nabla IM_p\Vert_p^p$, so that the following Proposition is self evident.
	\begin{prop}\label{finite energy and sobolev}
		Let $\mu$ be a charge on $\partial T$. Then $\mathcal{E}_p(\mu)<\infty$ if and only if $IM_p\in W^{1,p}(T)$.
	\end{prop}
	In terms of functions, Corollary \ref{same potentials cor} reads as follows.
	\begin{theorem}
		Let $g,h$ be $p$-harmonic functions in $W^{1,p}(T)$ satisfying (\ref{condition2}). If $g=h$, $c_p-$almost everywhere on $\partial T$, then $g\equiv h$ on $V(T)$.
	\end{theorem}
	\begin{proof}
		Glue together Corollary \ref{cor condition 2}, Proposition \ref{finite energy and sobolev} and Corollary \ref{same potentials cor}.
	\end{proof}
	
	If the tree $T$ is spherically symmetric, by Proposition \ref{spherical} we know that in the linear case $p=2$ we don't need the energy condition $g,h\in W^{1,2}(T)$ in the statement.
	
	\subsection{Boundary values} \label{Boundary}
	The uniqueness results we presented above apply to functions that admit boundary values $c_p-$almost everywhere. A priory is not obvious which functions have this property. Here we prove that in fact all the functions in the Sobolev spaces for which we have uniqueness results indeed admit boundary values $c_p-$almost everywhere. Other results of this kind can be found in \cite{Canton01}. To prove the existence of boundary values we follow an approach exploiting Carleson measures, which was already presented in \cite{Arcozzi08} for the linear case $p=2$. We include here the argumentation adapted for general Sobolev spaces for completeness.
	\bigskip
	
	We say that $\mu$ is a measure on $\overline{T}:=V(T)\cup\partial T$ if $\mu\big|_{V(T)}$ is a function on vertices and $\mu\big|_{\partial T}$ is a measure on the boundary. Observe that if $\mu\big|_{V(T)}=I^*(\mu|_{\partial T})$ then it defines a measure which is not finite.

	\begin{defn}
		We say that a Borel measure $\mu$ on $\overline{T}$ is a \textit{Carleson measure} for $W^{1,p}$ if there exists a constant $C(\mu)>0$ such that for all $g\in W^{1,p}$
		\begin{equation}\label{CM}
		\int_{\overline{T}}|g(\zeta)|^pd\mu(\zeta)\leq C(\mu)\Vert g\Vert_{1,p}^p.
		\end{equation}
	\end{defn}
	
	These measures have been widely studied and characterized (even in the weighted case), (see for example \cite{Arcozzi07}, \cite{Arcozzi05} and \cite{Arcozzi10}). In\cite{Arcozzi07} it is shown that condition (\ref{CM}) can be reformulated purely in terms of the measure $\mu$. In fact, it is shown that it is equivalent to
	\begin{equation}\label{CM2}
	\sum_{\beta\geq\alpha}\left(\int_{\overline{T}_\beta}d\mu\right)^p\leq C(\mu)\int_{\overline{T}_\alpha} d\mu,
	\end{equation}
	
	Denote by $\Vert \mu \Vert_{CM}$ the best possible constant in (\ref{CM2}), which for $\mu$ fully supported on $\partial T$ reduces to
	\begin{equation*}
	\Vert \mu \Vert_{CM}=\sup_{\alpha\in E(T)}\frac{\mathcal{E}_{p,\alpha}(\mu)}{M(\alpha)}.
	\end{equation*}
	Observe that if $\mu$ is the equilibrium measure for some set $E\subseteq \partial T$, by Lemma \ref{Rescaling} it follows that, for every edge $\alpha$, $\mathcal{E}_{p,\alpha}(\mu)/M(\alpha)\leq 1$, with equality for $\alpha=\omega$. Hence, $\Vert \mu \Vert_{CM}=1$ and $c_p(E)=\mu(E)/\Vert \mu \Vert_{CM}$. On the other hand, for any $\mu$ supported in $E\subseteq\partial T$ we have the bound $\mathcal{E}_p(\mu)\leq \Vert \mu \Vert_{CM}\mu(E)$, from which
	\begin{equation*}
	\frac{\mu(E)}{\Vert \mu \Vert_{CM}^{p-1}}\leq\frac{\mu(E)^p}{\mathcal{E}_p(\mu)^{p-1}}\leq c_p(E),
	\end{equation*}
	where the last inequality follows from the fact that the measure $\mu/\mathcal{E}_p(\mu)^{p-1}$ is admissible. We have derived the following expression of $p$-capacity in terms of Carleson measures of $W^{1,p}$ spaces
	\begin{equation}
	c_p(E)=\sup\left\lbrace\frac{\mu(E)}{\Vert\mu \Vert_{CM}}: \ \supp(\mu)\subseteq E\right\rbrace.
	\end{equation}
	
	The following proposition shows that the Fatou's set of a $W^{1,p}$ function differs from the boundary of the tree at most for a set of null capacity. The argument is taken by \cite{Arcozzi08}, where the result is proved for $ p=2 $.
	
	\begin{prop}
		Functions in $W^{1,p}$ have boundary values $c_p$-a.e. on $\partial T$
	\end{prop}
	
	\begin{proof}
		For $g\in W^{1,p}$ (that without loss of generality we normalize to $g(o)=0$), define the sequence of functions $g_n^*:=I\left( |\nabla g|\bigchi_{|\alpha|\leq n}\right)$. It is clear that $g_n^*$ is pointwise non-decreasing and it extends to the boundary by continuity, being eventually constant. By monotonicity we have that the function $g^*(\zeta)=\lim_n g_n^*(\zeta)$ is well defined for every $\zeta\in V(T)\cup\partial T$. Moreover, we have the uniform bound $\Vert g^*\Vert_{1,p}\leq \Vert g\Vert_{1,p}$. Now, let $\mu$ be a Carleson measure for $W^{1,p}$. By Fatou's Lemma we have
		\begin{equation*}
		\begin{split}
		\int_{\partial T} g^*(\zeta)^p d\mu(\zeta)&\leq \liminf_n \int_{\partial T} g_n^*(\zeta)^p d\mu(\zeta)\\
		&=\liminf_n\sum_{|\alpha|=n}\int_{\partial T_\alpha} g_n^*(\zeta)^p d\mu(\zeta)\\
		&=\sum_{|\alpha|=n}g(e(\alpha))^p M(\alpha)\\
		&\leq C(\mu)\Vert g\Vert_{1,p}.
		\end{split}
		\end{equation*}
		This implies that $g^*\in L^\infty(d\mu)$, and since $g^*$ is a bound for the radial variation of $g$ along geodesics, by Dominated Convergence Theorem we deduce that $g$ admits radial limit $\mu$-a.e. on $\partial T$ for every Carleson measure $\mu$. In particular, the equilibrium measure $\mu^E$ of the set $E=\partial T\setminus\mathcal{F}(g)$ is a Carleson measure since $\Vert \mu^E\Vert_{CM}=1$, from which follows that the radial limit exists $c_p$-a.e.
	\end{proof}

	\bibliography{DirichletOnTreesBib}
	\bibliographystyle{siam}

\end{document}